\documentclass{amsart}

\usepackage{graphicx}
\usepackage{amssymb,latexsym}
\usepackage{amsmath,amsfonts,amstext,bbm}
\usepackage[utf8]{inputenc}

\newtheorem{theorem}{Theorem}
\newtheorem{lemma}{Lemma}

\def\XXint#1#2#3{{\setbox0=\hbox{$#1{#2#3}{\int}$ }
\vcenter{\hbox{$#2#3$ }}\kern-.6\wd0}}

\begin{document}
\title[DISCRETE FRACTIONAL INTEGRAL OPERATORS WITH  QUADRATIC FORMS]
      {Discrete fractional integral operators with  binary quadratic forms as phase polynomials}
      
\author{Faruk Temur}
\address{Department of Mathematics\\
        Izmir Institute of Technology\\ Urla \\Izmir \\ 35430\\Turkey  }
\email{faruktemur@iyte.edu.tr}
\author{Ezgİ Sert}
\email{ezgisert@iyte.edu.tr}

\date{January  11, 2018}

\begin{abstract}
We give estimates on discrete  fractional integral operators along binary quadratic forms.   These operators  have been studied for 30 years starting with the investigations of Arkhipov and Oskolkov, but efforts have concentrated  on cases where the phase polynomial is   translation invariant or quasi-translation invariant. This work   presents    the first results for operators with neither translation invariant nor quasi-translation invariant phase polynomials.   
\end{abstract} 

\maketitle

\section{introduction}
Let $f:\mathbb{Z}^l \rightarrow \mathbb{C}$ be a function and $P:\mathbb{Z}^{k+l} \rightarrow \mathbb{Z}^l$ be a polynomial with integer coefficients. The operator 
\begin{equation}\label{fig}
\mathcal{I}_{\lambda}f(n)=\sum_{m\in \mathbb{Z}^k_*} \frac{f(P(m,n))}{|m|^{\lambda k}}
\end{equation}
where $\mathbb{Z}^k_*=\mathbb{Z}^k-\{0\}$ and $\lambda>0$,  is    called a discrete fractional integral. We call $P(m,n)$ the phase polynomial of $\mathcal{I}_{\lambda}f$. If this polynomial is $n-m$, then these discrete operators have essentially the same boundedness properties as those given by the Hardy-Littlewood-Sobolev theorem for their continuous counterparts. But if $P(m,n)$ involves a higher order term, then the discrete analogues satisfy boundedness results with more extensive ranges.   These operators and the related case of discrete singular Radon transforms
\begin{equation}\label{srg}
	\mathcal{R}f(n)=\sum_{m\in \mathbb{Z}^k_*} {f(P(m,n))}{K(m)}
\end{equation}
where $K$ is a Calderon-Zygmund kernel, 
have been studied extensively since \cite{ao}, and  efforts have  concentrated on the translation invariant case for which $P(m,n)=n-Q(m)$, and  the quasi-translation invariant case  for which the operators have the form
\begin{equation}\label{fiq}
\mathcal{J}_{\lambda}f(n,n')=\sum_{m\in \mathbb{Z}^k_*} \frac{f(n-m,n'-Q(m,n))}{|m|^{\lambda k}},
\end{equation}
 or in the case of Radon transforms
 \begin{equation}\label{srq}
 \mathcal{R}^*f(n,n')=\sum_{m\in \mathbb{Z}^k_*} {f(n-m,n'-Q(m,n))}{K(m)},
 \end{equation}
  with $f:\mathbb{Z}^{k+l} \rightarrow \mathbb{C}$, $Q:\mathbb{Z}^{k+k} \rightarrow \mathbb{Z}^l$.
 Translation invariance and  quasi-translation invariance make  the operators  amenable to Fourier analytic techniques, and they are studied as  multipliers, as exemplified by \cite{ao}. Over the last thirty years utilizing such tools as  multipliers, maximal functions, singular integrals, Hardy-Littlewood circle method a very extensive theory that is also connected to  ergodic theory  have been developed for these two cases. In this work we  treat cases that remain completely outside these efforts, and instead of  Fourier transform based methods we  utilize results from number theory on  representation of an integer by a polynomial, e.g. as a sum of two squares. We  use the number, structure and distribution of such representations in conjunction with appropriate decompositions to prove our results. 
 
 We would like to recall certain important milestones of   the study of discrete fractional integrals, a much more extensive account of these developments, as well as a general history of the study of discrete analogues in harmonic analysis can be found in \cite{lbp}. Observe that if in $\eqref{srg}$ we let $m\in \mathbb{Z}_*, K(m)=1/m, P(m,n)=n-Q(m) $, and regard  $n$ as a continuous variable in $\mathbb{R}$,      taking the Fourier transform in an appropriate sense would yield the multiplier
 \[m(\xi)=\sum_{m\in \mathbb{Z}_*} \frac{e^{-2\pi i \xi Q(m)}}{m}.\]
 The first result in the subject,  obtained by Arkhipov and Oskolkov in \cite{ao},  was that this $m(\xi)$ defines a  bounded function, which   implies in a straightforward way that the corresponding discrete singular Radon transform is a bounded operator on $l^2(\mathbb{Z})$.
 Then  Bourgain \cite{jb1,jb2} used the Hardy-Littlewood circle method to obtain results on the closely related maximal function
 \[\mathcal{M}f(n)=\sup_{r\geq 0}\frac{1}{r+1}\sum_{m=0}^{r}|f(n-Q(m))|.\]
The results and methods of   \cite{jb1,jb2}  inspired almost all of the later work done on the subject. For  discrete singular Radon transforms with translation invariant phase polynomials the work of Ionescu and Wainger \cite{iw}  greatly extended the known boundedness results by proving that if $m\in \mathbb{Z}^k,n\in \mathbb{Z}^l$  and $P(m,n)=n-Q(m)$, then $\eqref{srg}$ is a bounded operator on $l^p(\mathbb{Z}^l)$ for all $1<p<\infty$.  For quasi-translation invariant operators of the form $\eqref{srq}$, if $m,n\in \mathbb{Z}^k, n'\in \mathbb{Z}^l$,   and $Q:\mathbb{Z}^{k+k} \rightarrow \mathbb{Z}^l$ then  \cite{sw1}  proves that   $\mathcal{R}^*$ is a bounded operator on $l^2(\mathbb{Z}^{k+l})$. Also by \cite{imsw}, if $Q$ has degree at most 2 the operator  $\mathcal{R}^*$ is bounded  on $l^p(\mathbb{Z}^{k+l})$ for $1<p<\infty$.  As for the discrete fractional integrals   
 when $m,n\in \mathbb{Z}$ and $P(m,n)=n-m^2$, works of Stein and Wainger \cite{sw,sw2},  Oberlin \cite{ob}, Ionescu and Wainger \cite{iw} led to a complete understanding of the operator $\mathcal{I}_{\lambda}$,  and it is now known that this operator is bounded  from $l^p(\mathbb{Z})$ to $l^q(\mathbb{Z})$  if and only if 
 \[\frac{1}{q}\leq \frac{1}{p}-\frac{1-\lambda}{2},\ \ \ \ \ \  \ \ \ \ \frac{1}{q}<\lambda, \frac{1}{p}>1-\lambda.  \]
Unfortunately even for the cases $P(m,n)=n-m^s$ with $s>2$ sharp results remain out of reach. 
For the quasi-translation invariant operators $\mathcal{J}_{\lambda}$  too   there exist  almost sharp results for certain special cases, see for example \cite{lbp3}, while complete solution remains out of reach.  
   Various extensions and generalizations of the results exhibited in this short summary, as well as their connections and applications to ergodic theory  have been uncovered   in such works as  \cite{akm,mst1,mst2,lbp,lbp1,lbp2,lbp3}.

In this work we  prove results on operators of type \eqref{fig}, with $f\in l^p(\mathbb{Z})$ and $P(m,n)=q(m,n)=am^2+bmn+cn^2$, where both variables $m,n$ and coefficients     $a,b,c$ are integers.  Polynomials of this type are called integral binary quadratic forms, and they were studied intensely in 19th century by Gauss, Jacobi, Dirichlet and others.  Properties of $q$ depend greatly on  the discriminant   $\Delta(q):=b^2-4ac$. The form is called definite if $\Delta(q)<0$, and indefinite if $\Delta(q)>0$. It is called diagonal if $b=0$. We  state our first theorem.

\begin{theorem}
Let $f\in l^p(\mathbb{Z})$  where $1\leq p\leq \infty$. Let $q$ be a definite integral binary quadratic form with discriminant $\Delta(q)$. Then the  operator 
\[\mathcal{I}_{\lambda}f(n)=\sum_{m\in \mathbb{Z}_*} \frac{f(q(m,n))}{|m|^{\lambda }}\]
satisfies
\begin{equation}\label{t1}
\|\mathcal{I}_{\lambda}f\|_p \leq C_{p,\lambda,\Delta(q)}\|f\|_p.
\end{equation}
for $\lambda >1-p^{-1}$ when $p<\infty$,  and for $\lambda>1$ when $p=\infty$. This result is sharp in the following sense.

 For $p=1$ and $r\in \mathbb{N}$,  there is a  form $q$ and a function $f$ such that  $\|\mathcal{I}_{\log^r}f\|_1=\infty$, where 
\[\mathcal{I}_{\log^r}f(n)=\sum_{m\in \mathbb{Z}_*} \frac{f(q(m,n))}{\log^r (1+|m|)}.\]

For $1<p<\infty$, there is a form $q$ and a function $f$ such that $\|\mathcal{I}_{\lambda}f\|_p=\infty$ where $\lambda =1-p^{-1}.$

For $p=\infty$, there is a form $q$ and a function $f$ such that $\|\mathcal{I}_{\lambda}f\|_{\infty}=\infty$ where $\lambda =1.$
\end{theorem}

We would like to have a similar result for forms with nonnegative discriminant as well, and as will be stated in our second theorem this is possible when the discriminant is not a square number. Although the exponents in both theorems are exactly the same, the proofs are rather different, with the proof of the second theorem being much more delicate. Also we would like to emphasize that it is not possible to improve upon either result. We therefore state them separately. 

\begin{theorem}
Let $f\in l^p(\mathbb{Z})$  where $1\leq p\leq \infty$. Let $q$ be an indefinite   integral binary quadratic form with  nonsquare  discriminant $\Delta(q)$. Then the  operator 
\[\mathcal{I}_{\lambda}f(n)=\sum_{m\in \mathbb{Z}_*} \frac{f(q(m,n))}{|m|^{\lambda }}\]
satisfies
\begin{equation}\label{t2}
\|\mathcal{I}_{\lambda}f\|_p \leq C_{p,\lambda,\Delta(q)}\|f\|_p.
\end{equation}
for $\lambda >1-p^{-1}$ when $p<\infty$,  and for $\lambda>1$ when $p=\infty$. This result is sharp in the following sense.

 For $p=1$ and $r\in \mathbb{N}$,  there is a  form $q$ and a function $f$ such that $\|\mathcal{I}_{\log^r}f\|_1=\infty$, where 
\[\mathcal{I}_{\log^r}f(n)=\sum_{m\in \mathbb{Z}_*} \frac{f(q(m,n))}{\log^r (1+|m|)}.\]

For $1<p<\infty$, there is a form $q$ and a function $f$ such that $\|\mathcal{I}_{\lambda}f\|_p=\infty$ where $\lambda =1-p^{-1}.$

For $p=\infty$, there is a form $q$ and a function $f$ such that $\|\mathcal{I}_{\lambda}f\|_{\infty}=\infty$ where $\lambda =1.$
\end{theorem}

It will be clear in the next section that in both theorems the cases  $\lambda >1$ are trivial, and therefore by nontrivial estimates we mean   estimates with $\lambda\leq 1$. 
We will show in the next section that boundedness results with a similar range of $\lambda$ is never possible when the discriminant is a square number. But if it is a nonzero square number, with  some  additional assumptions on coefficients,  a set of weaker results may still be possible. We hope the methods we introduce here will allow progress on the more general  case of an arbitrary second degree integral polynomial  of two variables as the phase polynomial.  Even more interesting and challenging would be proving results for  forms of higher rank and degree, and then extending these results to arbitrary polynomials.  
   For  quadratic forms of higher rank proving such results with the methods of this work should be quite possible, for a great deal is known about  representations of integers by such forms. The case of  higher degree forms   seems to be very challenging, for  the theory of such forms is much less developed than that of quadratic forms. But it may still be possible to uncover some partial results. Connecting the methods and results of this work to the theory of maximal functions, and to ergodic theory would also be a very interesting task. The authors will pursue these questions further in their upcoming works.

  The article is organized as follows.  In the next section we establish the notation and terminology that will be used for the rest of this work, and we demonstrate that it is never possible to obtain results such as those given in our two theorems when the discriminant is a square number. Most importantly, by conducting a preliminary investigation of the problem we illustrate the approach that leads to the proofs of our theorems.  This approach uses the number, structure and distribution of representations of integers by quadratic forms. To understand these we study quadratic forms in section 3 from an analytic and geometric point of view. We  solve them over the field of real numbers, and  investigate the curves obtained from solutions. This allows us to prove two lemmas regarding  the distribution of representations of integers that will be of crucial importance in the proofs of our theorems. Then in section 4 we study quadratic forms from a number theoretic point of view using ideas and results obtained by Gauss, Jacobi, Dirichlet, and Pall. 
 We summarize certain facts discovered by them  regarding  the number and structure of representations of integers that will constitute the backbone  of our proofs. Finally in sections 5,6 we prove Theorem 1 and  Theorem 2 respectively. Proofs will rely heavily on sections 3,4.


\section{Preliminaries}

Henceforth we concentrate exclusively on integral binary quadratic forms, and reserve the notation $q(m,n)= am^2+bmn+cn^2$ to such forms with integer inputs $m,n$. We will use the letter $k$ to denote the values taken by $q$ on integer inputs $m,n$, that is $q(m,n)=k$. When this equation holds $(m,n)$ is called a representation of $k$ by $q$. We define the set of all representations of $k$ by $q$
\begin{equation}\label{rk}
R_{q,k}:=\{(m,n)\in \mathbb{Z}^2:q(m,n)=k\}.
\end{equation}
When the form $q$ is clear from the context we will drop it and write $R_k$. When we want to consider the same form $q$ on real numbers, we will use the notation $q(x,y), \ x,y\in \mathbb{R}.$ We similarly define for a real number $w$ the set 
\begin{equation}\label{sw}
S_{q,w}:=\{(x,y)\in \mathbb{R}^2:q(x,y)=w\},
\end{equation}
and similarly we will also denote this set by $S_w$. We will use the notation $\#E$ to denote the cardinality of a set $E$, and $|E|$ to denote its Lebesgue measure. Cardinality of $R_{k}$ and of certain intersections of $S_{w}$ with lines will be a chief concern of ours in the rest of this work.

  We reserve the notation $\Delta(q)$ to denote the discriminant of $q$, and when the form $q$ is clear from the context, we will only use $\Delta$. If $\Delta =b^2-4ac<0$, then clearly  $ac>0$, and thus for definite forms both $a,c$ are nonzero, and of the same sign.
 We observe that  $ ax^2+bxy+cy^2=w$ implies
\begin{equation}\label{pdf12}
 \begin{aligned} 
   4a^2x^2+4abxy+4acy^2=4aw,\ \ \ \ \text{and}
  \ \ \ \ 4acx^2+4bcxy+4c^2y^2=4cw,
 \end{aligned}
  \end{equation}
and completion of squares  yields
\begin{equation}\label{pdf19}
(2ax+by)^2-\Delta y^2=4aw , \ \ \ \ \text{and}
  \ \ \ \ (2cy+bx)^2-\Delta x^2=4cw.
\end{equation}
Therefore definite forms can take, even on the field of real numbers, only nonnegative or only nonpositive values depending on the sign of $a$. If $a>0$, the form is called positive definite, and if $a<0$, it is called negative definite. We observe that if $q$ is a positive definite form, then $-q$ is a negative definite one, and vice versa. Therefore it suffices to prove Theorem 1 only for positive definite forms, and henceforth we  consider only these. We note that when $\Delta>0$ the form $q$  takes values of both signs.

For two integers $u,v$  the notation $u|v$  means  $u$ divides $v$. 
A quadratic form $q(m,n)=am^2+bmn+cn^2$ is called primitive if the greatest common divisor   of $a,b,c$, denoted by $\gcd(a,b,c)$, is 1.
A representation $(m,n)$ of $k$ is called proper if $\gcd(m,n)=1$, and improper if it is not proper. In number theory it is mostly the case that studying  under the assumption of coprimality removes many obstacles. We will see that this indeed is the case for the study of representations of integers by quadratic forms, and therefore before understanding the set of all representions it is more fruitful to study the set of proper representations given by 
\begin{equation}\label{rk}
R'_{q,k}:=\{(m,n)\in \mathbb{Z}^2:q(m,n)=k,\  \gcd(m,n)=1\}.
\end{equation}

We now   show that an analogue of Theorem 1 and Theorem 2 for quadratic forms with square discriminant is never possible. Let $q(m,n):=am^2+bmn+cn^2$ be a form with  $\Delta(q)$   perfect square, that is  $\Delta=d^2, \ d\in \mathbb{N}\cup \{0\}$. We observe that if $c\neq 0$, then     the points $(2cj,(-b+ d)j), \ j \in\mathbb{Z}$ are distinct elements of the integer lattice $\mathbb{Z}^2$, and when plugged in $q$ they give $0$.  Thus, for 
\[f(k)=\begin{cases}
1  & \text{if}\ \  k=0\\
0 & \text{otherwise}
\end{cases}\] 
we have
\begin{equation*}
\begin{aligned}
\|\mathcal{I}_{\lambda}f\|_{1}= \sum_{n\in \mathbb{Z} }\left|  \sum_{m\in \mathbb{Z}_*} \frac{f(q(m,n))}{|m|^{\lambda }}     \right| &= \sum_{n\in \mathbb{Z} } \sum_{m\in \mathbb{Z}_*} \frac{f(q(m,n))}{|m|^{\lambda }}\\ &\geq \sum_{j\in \mathbb{Z}_*}\frac{f(q(2cj,(-b+ d)j))}{|2cj|^{\lambda}}\\ & = \sum_{j\in \mathbb{Z}_*}\frac{1}{|2cj|^{\lambda}},     
\end{aligned}
\end{equation*}
and this diverges for $\lambda \leq 1$. If $c=0$, then $q(m,n)=am^2+bmn$, and $(bj,-aj), \ j\in\mathbb{Z}$ are elements of the integer lattice $\mathbb{Z}^2$, and when plugged in $q$ they give 0. Assuming $b\neq 0$,   with the same $f$ as above we have
 \begin{equation*}
 \begin{aligned}
 \|\mathcal{I}_{\lambda}f\|_{1}= \sum_{n\in \mathbb{Z} } \sum_{m\in \mathbb{Z}_*} \frac{f(q(m,n))}{|m|^{\lambda }} &\geq \sum_{j\in \mathbb{Z}_*}\frac{f(q(bj,-aj))}{|2bj|^{\lambda}}\\ & \geq \sum_{j\in \mathbb{Z}_*}\frac{1}{|2bj|^{\lambda}},     
 \end{aligned}
 \end{equation*}
and this diverges for $\lambda \leq 1$. If $b,c$ are both zero, then  $q(m,n)=am^2$. In this case we take  
\[f(k)=\begin{cases}
1  & \text{if}\ \  k=a\\
0 & \text{otherwise,}
\end{cases}\] 
and consider the points $(1,j), \ j\in \mathbb{Z}$. We have
 \begin{equation*}
 \begin{aligned}
 \|\mathcal{I}_{\lambda}f\|_{1}= \sum_{n\in \mathbb{Z} } \sum_{m\in \mathbb{Z}_*} \frac{f(q(m,n))}{|m|^{\lambda }} \geq \sum_{j\in \mathbb{Z}}\frac{f(q(1,j))}{1^{\lambda}} =\infty.
 \end{aligned}
 \end{equation*}
Therefore,  for no quadratic form of square discriminant is a boundedness result on $l^1(\mathbb{Z})$ possible   when   $\lambda \leq 1$.
 The same constructions impose the condition $\lambda>p^{-1}$ for  boundedness results on $l^p(\mathbb{Z}),\ 1<p<\infty$. A counterexample similar to those we provide in sections 5,6 would further impose the condition $\lambda > 1-p^{-1}$,  and some restrictions on coefficients. Further yet,  it is possible to   demonstrate  that the discriminant must be nonzero.
 When all these conditions are met,   we may have a  set of nontrivial  estimates on these spaces.  
As stated in the introduction this issue will be investigated later.

In what follows we  conduct a preliminary investigation of the basic case of  estimates on $l^1(\mathbb{Z})$. As will later be seen in the proofs of both Theorem 1 and Theorem 2 this is the most important case. This investigation  describes the starting point of our arguments, and we will return to it in the proofs of both our theorems. We 
let $f\in l^1(\mathbb{Z}),$ and let
\begin{equation*}
\mathcal{I}_{\lambda}f(n)=\sum_{m\in \mathbb{Z}_*} \frac{f(q(m,n))}{|m|^{\lambda }}
\end{equation*}
be our operator with an arbitrary  quadratic form $q(m,n)=am^2+bmn+cn^2.$ We emphasize that we put no restrictions regarding $\Delta(q).$ We have
\begin{equation}\label{eq1}
\|\mathcal{I}_{\lambda}f\|_{l^1(\mathbb{Z})}= \sum_{n\in \mathbb{Z} }\Big|  \sum_{m\in \mathbb{Z}_*} \frac{f(q(m,n))}{|m|^{\lambda }}     \Big| \leq \sum_{n\in \mathbb{Z} } \sum_{m\in \mathbb{Z}_*} \frac{|f(q(m,n))|}{|m|^{\lambda }}.     
\end{equation}
We define  the sets 
\begin{equation}
A_k:=\{(m,n)\in \mathbb{Z}_* \times \mathbb{Z}: q(m,n)=k\}
\end{equation}
 for each $k\in \mathbb{Z}$. These sets form a partition of $\mathbb{Z}_* \times \mathbb{Z}$, for clearly every element of   $ \mathbb{Z}_* \times \mathbb{Z}$ belongs to one of these sets $A_k$, and if $k\neq l$ then $A_k\cap A_l = \emptyset$. Therefore the last term of $\eqref{eq1}$  satisfy
\begin{equation}\label{eq2}
= \sum_{(m,n)\in \mathbb{Z}_* \times \mathbb{Z} } \frac{|f(q(m,n))|}{|m|^{\lambda }}     = \sum_{k\in \mathbb{Z} }|f(k)| \Big[ \sum_{(m,n)\in A_k} \frac{1}{|m|^{\lambda }}\Big].     
\end{equation}
 Since  $f\in l^1(\mathbb{Z})$, if we could show that 
 \begin{equation}\label{eq3}
\sum_{(m,n)\in A_k} \frac{1}{|m|^{\lambda }}\leq C     
 \end{equation}
 for all $k$, with a constant $C$ independent of $k$, we would have the desired estimate
$
 \|\mathcal{I}_{\lambda}f\|_{l^1(\mathbb{Z})}  
\leq C
\|f\|_{l^1(\mathbb{Z})}. 
$
Therefore  the problem of  estimates on $l^1(\mathbb{Z})$ essentially reduces to understanding the quantity
\begin{equation}\label{eq5}
\sum_{(m,n)\in A_k} \frac{1}{|m|^{\lambda }},
 \end{equation}
and this clearly is about the number of representations of $k$ by the form $q$, and the distribution and structure of these representations. If the number of representations is small, and if all except a fixed number of them have large first coordinates then this can be bounded by an absolute constant. We will use this idea to prove Theorem 1. If the number of representations is large or even infinite, then there still is hope if we can, by deeper study, show that the first coordinates of these representations grow fast. Incorporating this idea into the proof of Theorem 1 will lead us to  the proof of  Theorem 2. The next section is devoted to showing that with a fixed number of exceptions the  first coordinates of the representations are large. We will deduce this from a geometric and analytic investigation of quadratic forms conducted by solving them over the field of real numbers.  In section 4, using the classical theory of binary quadratic forms developed mostly in 19th century, we will study the number and structure of representations.

We note that the term in \eqref{eq5} is always finite if $\lambda >1$ and for a fixed $m$, the set $A_k$ contains at most a constant number of representations $(m,n)$. It will be seen in section 3 that this latter condition is satisfied by all forms under consideration in our two theorems.   It also   holds   for many forms of square discriminant.


 \section{Analysis and geometry of  Quadratic Forms}
In this section we  consider the representation  problem  over the field of real numbers, where it is much easier to understand. We  investigate the set $S_{q,w}$ as given in \eqref{sw}. This leads to quite different outcomes for positive definite forms and indefinite forms of nonsquare discriminant, therefore we investigate them separately. Further information on many of the issues discussed in this section can be found in analytic geometry books such as \cite{ag1,ag2,ag3}.

 We first discuss positive definite forms.
  Let $q(x,y)=ax^2+bxy+cy^2$  be positive definite, therefore $a,c$  are positive, and $-2\sqrt{ac}< b< 2\sqrt{ac}$. By (\ref{pdf12},\ref{pdf19}), the set 
    $S_w$ is empty if $w<0$, and contain only the origin if  $w=0$. Therefore to observe nontrivial cases we assume $w>0$. In this case $S_{w}$ is an ellipse centered at the origin, and as such it is a closed plane curve enclosing a strictly convex region. A line can intersect it at at most two points owing to this  strict convexity. 
    It always contains the points
      $(\pm\sqrt{w/a},0),(0,\pm \sqrt{w/c})$ regardless of the value $b$ takes, and   by convexity  the paralellogram with these points as  vertices  lies inside it. We  observe that the lines that contain the sides of this  parallelogram essentially dictate the behavior of $S_w$ as $b$ changes. Decreasing $b$ from zero to $-2\sqrt{ac}$ elongates the set  in the direction of the line  $y=\sqrt{a/c}x$, and incresing $b$  to $2\sqrt{ac}$ elongates it  in the direction of the line  $y=-\sqrt{a/c}x$. At limits the ellipse turns into the lines given by the sides of the parallelogram.

      We can write this curve as graphs of two functions: solving the equation $ax^2+bxy+cy^2=w$ for $y$ is possible  if and only if  $x^2\leq -4cw/\Delta$, and gives
        \begin{equation}\label{func1}
        y=f_1(x)=\frac{-bx+\sqrt{\Delta x^2+4cw}}{2c}, \ \ \ \ y=f_2(x)=\frac{-bx-\sqrt{\Delta x^2+4cw}}{2c}.
        \end{equation}
       The graph of $f_1$ lies above that of $f_2$  except at the endpoints, and they meet at the  endpoints. 
       On their domain of definition  $f_1$ is concave and   $f_2$ is convex.  
       
       We  state a lemma that will be crucial for the proof of our first theorem. It shows that except for  a fixed number of them, representations of integers by positive  definite forms have large first entries. The proof  essentially boils down to  the degree of the  form $q$.

\begin{lemma}
Let $q(x,y)= ax^2+bxy+cy^2$ be a positive definite binary quadratic form, and let $k$ be an integer. Then $q(x,y)=k$ has at most 4 solutions $(x,y)\in \mathbb{Z}^2$ satisfying 
\begin{equation}\label{lem1} 
|x| \leq \frac{|k|^{1/4}}{\sqrt{-\Delta(q)}}.
\end{equation}
\end{lemma}
\begin{proof} 
 The statement is clear if $k\leq 0$, we therefore assume $k$ to be positive. The solutions we are looking for lie on the  graphs of the functions $f_1,f_2$ in \eqref{func1}. Owing to  the condition \eqref{lem1},
any of these solutions can lie on only one of these graphs. 
The lines 
\[l_1(x)=-\frac{b}{2c}x+\sqrt{\frac{k}{c}}, \quad \quad l_2(x)=-\frac{b}{2c}x-\sqrt{\frac{k}{c}} \]
are respectively tangent to $f_1,f_2$ at $x=0$.
   We will  prove that  $f_1,f_2$  stay very close to these lines for $x$ satisfying \eqref{lem1}. The differences   $|f_i(x)-l_i(x)|, \ i=1,2$ for such $x$  are bounded   by
\begin{equation}
\begin{aligned}
\sqrt{\frac{k}{c}}-\frac{\sqrt{\Delta x^2+4ck}}{2c}&=  \Big(\frac{k}{c}\Big)^{1/2}-\Big(\frac{{\Delta x^2}}{4c^2} +\frac{k}{c}\Big)^{1/2} \\ &=-\frac{{\Delta x^2}}{4c^2} \cdot \Big[\Big(\frac{{\Delta x^2}}{4c^2} +\frac{k}{c}\Big)^{1/2}+  \Big(\frac{k}{c}\Big)^{1/2} \Big]^{-1} \\ &\leq -\frac{{\Delta x^2}}{4c^2} \cdot \Big[ \frac{3}{2} \Big(\frac{k}{c}\Big)^{1/2} \Big]^{-1} \\ & \leq \frac{1}{6c^{3/2}}.
\end{aligned}
\end{equation}
Therefore our solutions satisfying $y=f_i(x)$ lie inside the set
$S_i:=\{(x,y)\in \mathbb{R}^2:  |y-l_i(x)|  \leq {1}/{6c^{3/2}}    \} $
for $i=1,2.$

However, if $(m,n)\in \mathbb{Z}^2$, then $2cn+bm=j\in \mathbb{Z} $. Thus 
\[n=-\frac{b}{2c}m+\frac{j}{2c}.\]
Therefore  if we consider the collection of parallel lines 
\[\Big\{(x,y)\in \mathbb{R}^2:  y=-\frac{b}{2c}x+\frac{j}{2c}\Big\},\]
every element $(m,n)\in \mathbb{Z}^2$ lies on exactly one of these lines. But the sets $S_1,S_2$ each can contain at most one line from this collection. Therefore we have at most 4 solutions.

\end{proof}

We now investigate the case of indefinite forms of nonsquare discriminant.
We assume $\Delta(q)=b^2-4ac > 0 $ to be  nonsquare, therefore both $a,c$ are nonzero. We will assume $c>0$, and investigate the set $S_w$ for any  real number $w$, clearly the case  $c<0$ follows from this. As will be seen the sign of $w$ mostly determines the geometry of the set $S_w$. 

Let $w=0$. Then for a fixed $x$ the equation
$ax^2+bxy+cy^2=0$ is satisfied if and only if 
\begin{equation}\label{lines}
y=j_1(x)=\frac{-b+\sqrt{\Delta}}{2c}x,  \quad \quad \quad y=j_2(x)=\frac{-b-\sqrt{\Delta}}{2c}x. 
 \end{equation} 
Therefore $S_w$ consists of these two lines. When $b=0$ these lines have slopes that add up to zero, as  $b$ decreases they turn counterclockwise, and as $b$ increases they turn clockwise. We will see that these lines govern the behavior of $S_w$ even for $w$ nonzero.

Let $w>0$. In this case $S_w$ is a hyperbola centered at the origin with the lines in \eqref{lines} as asymptotes. As such it is a smooth plane curve. The graphs of the functions
\begin{equation}\label{func2}
y=g_1(x)=\frac{-bx+\sqrt{\Delta x^2+4cw}}{2c},  \quad \quad \quad y=g_2(x)=\frac{-bx-\sqrt{\Delta x^2 +4cw}}{2c},  
\end{equation}
give the two disconnected  subsets of the hyperbola, with $g_1$ being above both asymptotes and $g_2$ being below both of them.

 Let $w<0$. In this case our set is the conjugate of the hyperbola we would obtain for $-w$. The two disconnected subsets of $S_w$ lie between the asymptotes. The equation $ax^2+bxy+cy^2=w$ can be  satisfied only if $x^2\geq -4cw/\Delta.$

We  investigate how many times a line can intersect $S_w$. If $w=0$, then  $S_w$ itself is union of two lines intersecting at the origin, therefore a different line can intersect it at most twice. 
For $w\neq 0$ we will consider the equation for a generic line and plug it into the equation given by our form. Any line in $\mathbb{R}^2$ has an equation   of the form either $y=ux +v$ or $x=uy+v$ where $u,v$ are real numbers.   Plugging these equations into $ax^2+bxy+cy^2=w$ gives respectively
\begin{equation}\label{linin}
\begin{aligned}
x^2(cu^2+bu+a)+x(2cuv+bv)+cv^2&=w,  \\ y^2(au^2+bu+c)+y(2auv+bv)+av^2&=w.
\end{aligned} 
\end{equation}
 The  coefficients of $x^2$ and $x$, or $y^2$ and $y$ cannot both be  zero since $w\neq 0$ and $\Delta>0$. Therefore we have at most 2  intersections. With this information at hand we proceed to state an analogue of our first lemma for the case of indefinite forms of nonsquare discriminant.

\begin{lemma}
Let $q(x,y)= ax^2+bxy+cy^2$ be an indefinite form of nonsquare discriminant, and let $k$ be an integer. Then $q(x,y)=k$ has at most 4 solutions $(x,y)\in \mathbb{Z}^2$ satisfying 
\begin{equation}\label{lem2} 
|x| \leq \frac{|k|^{1/4}}{\sqrt{\Delta(q)}}.
\end{equation}
\end{lemma}
\begin{proof}
We  assume $c>0,$ the case $c<0$ follows from this by considering $-q$ and $-k$. We proceed in three cases. We first assume $k=0$. In this case \eqref{lem2} implies $x=0$, and  since $c\neq 0$, this implies $y=0$. Thus   the origin is the only solution.

We  assume $k>0$. Then our solutions lie on the graphs of the  functions in \eqref{func2},
 and since these graphs do not intersect, each solution can lie only on one of these graphs. The lines
 \begin{equation*}
 y=l_1'(x)=-\frac{b}{2c}x+\sqrt{\frac{k}{c}},  \quad \quad \quad y=l_2'(x)=-\frac{b}{2c}x-\sqrt{\frac{k}{c}}, 
  \end{equation*} 
 are tangent  respectively to $g_1,g_2$ at $x=0$.  The difference  between  $g_1,g_2$ and these lines 
 for a fixed $x$  satisfying \eqref{lem2} can be shown   by the same arguments as in Lemma 1 to satisfy
 \begin{equation}
 \begin{aligned}
 \frac{\sqrt{\Delta x^2+4ck}}{2c}-\sqrt{\frac{k}{c}}  \leq \frac{1}{8c^{3/2}}.
 \end{aligned}
 \end{equation}
 Then   arguments similar to those in Lemma 1 allow us to reduce the problem  to  the number of intersections $S_k$ can have with two lines, and  our investigation above lets us conclude that there can be at most 4 solutions.
 
 Finally let $k<0$. In this case for  any element $(x,y)\in S_k$ we must have $x^2\geq -4ck/\Delta$, but this contradicts  \eqref{lem2}, therefore there is no solution.

\end{proof}


\section{Arithmetic of quadratic forms}

 In this section we concentrate on understanding the set  $R_{q,k}$. This  is a topic that attracted great interest in the 19th and early 20th centuries. Through works of Gauss, Jacobi, Dirichlet, Pall and others we now know that  properties of $R_k$ are ultimately tied  to the theory of quadratic residues, and to the  automorphs of $q$.  We will derive this connection, and describe its use. We will mostly use the terminology and notation of \cite{ld}, and also recommend that work for a  complete exposition of binary quadratic forms that starts from very basics of number theory. For a more recent treatment see \cite{db}.

Let $q(m,n)=am^2 +bmn+cn^2$ be a  binary quadratic form.  We will denote $q$ by $[a,b,c]$, and associate to it the matrix
\begin{equation*}
[q]:=\begin{bmatrix}
	  2a     & b \\
	 b & 2c  
\end{bmatrix}.
\end{equation*}
Clearly then   $\Delta(q)=-\det[q]$.   We consider  a matrix $[T]$ of  integral entries given by
\begin{equation*}
[T]:=
\begin{bmatrix}
	    \alpha   & \beta \\
	\gamma & \delta  
\end{bmatrix},
\end{equation*}
and the linear map given by this matrix   
\begin{equation*}
\begin{bmatrix}
	   m  \\
	n  
\end{bmatrix}=
\begin{bmatrix}
	    \alpha   & \beta \\
	\gamma & \delta  
\end{bmatrix}
\begin{bmatrix}
	    M \\
        N  
\end{bmatrix}.
\end{equation*}
This matrix turns our form $q$ into   $Q(M,N)=AM^2+BMN+CN^2$,  if we replace $m,n$ by their equivalents $\alpha M+\beta N, \gamma M+\delta N$ 
as given by the matrix multiplication. We easily calculate $A,B,C$ to find
\begin{equation*}
A=a\alpha^2+b\alpha\gamma +c\gamma^2, \ \ B=2a\alpha\beta +b(\alpha\delta + \beta\gamma) +2c\gamma\delta,  \ \ C=a\beta^2+b\beta\delta +c\delta^2.
\end{equation*}
We denote the map thus induced by a matrix $[T]$ with $T$,  thus we have $Tq=Q$.  The action of this map in terms of matrices is given by 
\begin{equation}\label{revc41}
[Q]=[T]^*[q][T],
\end{equation}
where $[T]^*$ stands for the transpose of $[T].$
 If the determinant of $[T]$ is $\pm 1$, then  $[T]^{-1}$ also have integral entries and induces a map.  It follows from \eqref{revc41} that in this case $T$ is a bijection from the set of integral binary quadratic forms to itself, with the  inverse being the map induced by $[T]^{-1}.$   Then both the integers represented, and the number of representations for each such integer are the same for the forms $q$ and $Q=Tq.$ Furthermore their  discriminants are the same.  If the determinant is 1 we say that these two forms are properly equivalent, if it is $-1$ then we say that they are improperly equivalent. We sometimes just use the term equaivalent to mean properly equivalent. Clearly proper equivalence is an equivalence relation, so it partitions forms of a given discriminant into equivalence classes. In the theory of quadratic forms class of a form means the equivalence class with respect to this relation. If $[T],[S]$ are  two  matrices of determinant 1 that   transform  $q$ to $Q$,  then
 \[[q]= [S]^{-1*}[Q][S]^{-1}=[S]^{-1*}[T]^*[q][T][S]^{-1}=([T][S]^{-1})^*[q][T][S]^{-1}.\] 
If a matrix of determinant $1$ fixes a form, it is called an automorph of that form. Thus $[T][S]^{-1}$ is an automorph of $q$. If we denote this automorph by $[A]$, we obtain the result $[T]=[A][S]$, and thus matrices of determinant 1 that transform  $q$ to $Q$ are exactly those given by finding just one such matrix, and multiplying it with the automorphs of $q$.

Let $g\in \mathbb{N}$ satisfy $g^2|k.$  Then  $(m,n)\mapsto (m/g,n/g)$ gives a bijection between the representations $(m,n)$ of $k$ with $g=\gcd(m,n)$ and the proper representations of $k/g^2$. Therefore if $k\neq 0$, then
 \begin{equation}\label{union}
R_k=\bigcup_{g^2|k}g\cdot R'_{k/g^2}.
\end{equation}
If $k=0$,  the set $R_k$ consists of the union above and the origin.

 We first consider  $R_k$  for the case $k=0$.
 For definite forms    the only representation is $(0,0)$, this is clear from the formulas (\ref{pdf12}, \ref{pdf19}). For the form with all coefficients zero, every element of $\mathbb{Z}^{2}$ is a representation. When the discriminant is nonnegative and  the form has at least one nonzero coefficient, the representations are intersections of one or two lines given by linear factors of the form with $\mathbb{Z}^{2}$. When the discriminant is nonsquare  by  \eqref{lines} the origin is the only intersection. For a square discriminant, for each line they are   given by $(jm,jn), \ j\in\mathbb{Z}$ where $\gcd(m,n)=1.$  
 
    We now  assume $k\neq 0$, and concentrate on understanding the proper representations $R'_k$. If there is a pair $(\alpha,\gamma)$ such that  $q(\alpha,\gamma)=k$ and $\gcd(\alpha,\gamma)=1$, then we can find a pair $\beta,\delta$ with
$\alpha\delta-\gamma\beta=1$. If $\beta',\delta'$ is another pair with the same property then
$\alpha\delta-\gamma\beta=\alpha\delta'-\gamma\beta'$, and thus $ \alpha(\delta'-\delta)= \gamma(\beta'-\beta). $
Since $\gcd(\alpha,\gamma)=1$, at least one of these is nonzero,  assume $\alpha$ is nonzero. Then $\alpha$ divides $\beta'-\beta$. So we must have $\beta'=\beta+t\alpha$ for some uniquely determined  integer $t$. This then implies $\alpha(\delta'-\delta)=\gamma t\alpha$, and since $\alpha$ is not zero, we must have $\delta'=\delta+t\gamma$. If $\gamma$ is nonzero we run the same argument to conclude that in any case we have a uniquely determined integer $t$ such that 
\begin{equation}\label{bd}
\beta'=\beta+t\alpha, \ \ \ \ \ \delta'=\delta+t\gamma. 
\end{equation}
Conversely   any integer $t$ determines unique numbers  $\beta',\delta'$ given by \eqref{bd} that satisfy $\alpha\delta'-\gamma\beta'=1.$
 The matrix 
\begin{equation}\label{matr1}
\begin{bmatrix}
	    \alpha   & \beta \\
	\gamma & \delta  
\end{bmatrix}\end{equation}
then have determinant 1, and transforms  $q$ to an equivalent form $[k,u,v]$, where
\[k=a\alpha^2+b\alpha\gamma +c\gamma^2, \ \ u=2a\alpha\beta +b(\alpha\delta + \beta\gamma) +2c\gamma\delta,  \ \ v=a\beta^2+b\beta\delta +c\delta^2.\]
 On the other hand, the matrix
\[\begin{bmatrix}
	    \alpha   & \beta' \\
	\gamma & \delta'  
\end{bmatrix}\]
transforms $q$ to the form $[k,u',v']$ with 
\[u'=2a\alpha\beta' +b(\alpha\delta' + \beta'\gamma) +2c\gamma\delta',  \ \ \ v'=a(\beta')^2+b\beta'\delta' +c(\delta')^2.\]
We observe that replacing $\beta'=\beta+t\alpha$ and $\delta'=\delta+t\gamma$ we obtain
\begin{equation}\label{cong6}
\begin{aligned}
u'&=2a\alpha\beta' +b(\alpha\delta' + \beta'\gamma) +2c\gamma\delta'\\
&=2a\alpha(\beta +t\alpha) +b[\alpha(\delta+t\gamma) + (\beta +t\alpha)\gamma] +2c\gamma(\delta+t\gamma)\\
&=2a\alpha\beta +b(\alpha\delta + \beta\gamma) +2c\gamma\delta + 2t[a\alpha^2 +b\alpha\gamma +c\gamma^2]\\
&=u+2tk.
\end{aligned}
\end{equation}
Therefore there is a unique choice of the pair $\beta,\delta$ for which  $0\leq u< 2|k|$.  Since the discriminant remains fixed, we also have the relation $u^2-4kv=\Delta(q)$.   Since $k\neq 0$, given $u$ there can be a unique  $v$ solving this equation, therefore  there is a one-to-one and onto correspondence between solutions $u,v$ of these two  equations and solutions  $u$ of 
\begin{equation}\label{cong5}
u^2\equiv \Delta(q) \ \  (\text{mod}\ 4|k|), \quad \text{and}  \quad 0\leq u< 2|k|.
\end{equation}
 Therefore, going forward we will employ these latter equations that do not burden us with the  coefficient $v$ that plays no role in this theory.

 We sum up what we have uncovered  as follows: for a representation $(\alpha,\gamma)$, by choosing $\beta,\delta$ as prescribed, we obtain a  unique  matrix   \eqref{matr1}
of determinant 1 that  transforms $q$ to an equivalent form $[k,u,v]$ satisfying the equations \eqref{cong5}.
We  thus see that by this process   representations are mapped injectively to   matrices,  which in turn are mapped to   quadratic forms satisfying a congruence. We observe however that the mapping of  matrices to  forms is definitely not injective, for the matrix
\[\begin{bmatrix}
	    -\alpha   & -\beta \\
	-\gamma & -\delta  
\end{bmatrix}\]
too would  transfrom $q$ to $[k,u,v]$. Indeed, as we have observed,  matrices of determinant $1$ transforming $q$ to $[k,u,v]$ are exactly those obtained   by finding just one such matrix and multiplying it with the automorphs of $q$. 
Conversely,  when  we obtain  a form $[k,u,v]$ satisfying \eqref{cong5}, there are two possibilities: either it is not equivalent to $q$, in which case there is no corresponding representation, or it is equivalent to $q$, whence by finding one matrix  transforming $q$ to $[k,u,v]$, and multiplying it with the automorphs of $q$ we obtain all such matrices, and the first columns of these matrices give the representations corresponding to $[k,u,v]$. We repeat  this for every form satisfying  \eqref{cong5},  and obtain  a number of matrices. Of these matrices no two can be the same, for if two matrices map $q$ to different forms they must be different, and if they map $q$ to the same form, this would imply that the automorphs we used to obtain them are the same. Furthermore, no two of these matrices can have the same first column, for this is forbidden by  \eqref{cong6}. Thus the first columns of  these matrices are exactly the desired proper representations. Hence to uncover the number and structure of  proper representations we need to understand two issues:  how many forms  $[k,u,v]$ satisfying conditions \eqref{cong5} exist, and   how many automorphs of $q$ exist.
  If we could find answers to these two issues, by simply multiplying these answers we would get an upper bound for $R'_k$. Further, if   we knew that every solution of \eqref{cong5}
is equivalent to  $q$, we  could find the number  of proper representations exactly. We note however that  we will never need such exact knowledge of representations in this work, we will only need  upper bounds on their number, lower bounds for their first entries, and  their structure under the assumption that they exist. 

 The answer to the first issue is purely about quadratic residues, and depends on the relation between $\Delta$ and  $k$.  For the case $\gcd(\Delta, k)=1$ the answer was given  by Dirichlet and as a part of the classical theory of binary quadratic forms can be found in any basic number theory book, such as \cite{ld}. The  complete answer  was given by G. Pall in \cite{gp1,gp2}. We here lay forth a summary of their studies. We define for $t\in \mathbb{Z},s\in \mathbb{N}$ the function
 \begin{equation}
 \Gamma_t(s)=\# \{u:   u^2\equiv t \ (\text{mod} \ 4s), \quad 0 \leq u<2s\}.
 \end{equation}
Since $u^2\equiv (u+2ls)^2 \ (\text{mod} \ 4s)$, we have
  \begin{equation}
   \Gamma_t(s)=\frac{1}{2}\# \{u:   u^2\equiv t \ (\text{mod} \ 4s)\}.
   \end{equation}
If $t$ gives the remainders $2,3$ when divided by $4$, the definition of our function implies a perfect square giving these remainders when divided by $4$, therefore the function is uniformly zero for all values of $s$ in these cases. Therefore we are left with the cases when $t$ gives the remainders $0,1$, and since  all discriminants  give these remainders when divided by $4$,  these  are the cases  that are of importance to us. We immediately observe that for such $t$ we have $\Gamma_t(1)=1$. For larger $s$ we consider the prime factorization  $s=p_0^{a_0}p_1^{a_1}\ldots p_j^{a_j}$ with $p_0=2,\ a_0\geq  0$ and $a_i>0$ for $1\leq i\leq j$. Then by Theorem 16 of \cite{ld}
\begin{equation}
\begin{aligned}
\Gamma_t(s)&=\frac{1}{2}\# \{u:   u^2\equiv t \ (\text{mod} \ p_0^{a_0+2}p_1^{a_1}\ldots p_j^{a_j})\}\\ 
&= \frac{1}{2}\# \{u:   u^2\equiv t \ (\text{mod} \ p_0^{a_0+2})\}\prod_{i=1}^j \# \{u:   u^2\equiv t \ (\text{mod} \ p_i^{a_i})\}.
\end{aligned}
\end{equation}
Since $\Gamma_t(1)=1$, again by the same theorem for $i>0$ we have 
\begin{equation}
\begin{aligned}
\Gamma_t(p_i^{a_i})&=\frac{1}{2}\# \{u:   u^2\equiv t \ (\text{mod} \ 4p_i^{a_i})\}\\ 
&= \frac{1}{2}\# \{u:   u^2\equiv t \ (\text{mod} \ 4)\} \# \{u:   u^2\equiv t \ (\text{mod} \ p_i^{a_i})\}\\ &=\# \{u:   u^2\equiv t \ (\text{mod} \ p_i^{a_i})\}.
\end{aligned}
\end{equation}
Therefore
\begin{equation}
\Gamma_t(s)=\prod_{i=0}^j \Gamma_t(p_i^{a_i}).
\end{equation}
Thus it remains to compute $\Gamma_t(p^{a})$ for prime $p$ and positive $a$. When $t=0$ we easily see that  $\Gamma_t(p^{a})=p^{\lfloor a/2\rfloor }$. For $t$ nonzero, the results  depend very much on divisibility relations between $p$ and $t$, and can be obtained via use of Theorem 17 of \cite{ld} together with basic arguments of modular arithmetic. The full results are tabulated  in section 3 of \cite{gp1} in nine cases. We remark that although in that work $t$ is assumed to be negative, the results therein  depend   not on the  sign of $t$, but on the divisibility relations stated for each case, and thus follow through for positive $t$ as well. We observe that  the common bound  $\Gamma_t(p^{a}) \leq 2p^{\lfloor c/2\rfloor}$, where $c$ is the power of $p$ as a factor of $t$, is true for all cases. Therefore $\Gamma_t(s) \leq d(s)\sqrt{|t|}$, where the function $d$ gives the number of positive divisors. It is known, see \cite{hw}, that for every $\varepsilon>0$ we have $d(s)\leq C_{\varepsilon}s^{\varepsilon}$. Therefore for every  $\varepsilon>0$  we have
\begin{equation}\label{revc42}
\Gamma_t(s) \leq C_{\varepsilon}s^{\varepsilon}\sqrt{|t|} .
\end{equation}

We  turn to the second question. If  $\Delta<0$, or if $\Delta >0$ is a square  then the form has at most 6 automorphs, see  \cite{ld,gp1,gp2}. 
Thus  from  \eqref{revc42} we obtain 
\begin{equation*}
\#R'_{k}\leq C_{\varepsilon}|k|^{\varepsilon}\sqrt{|\Delta|}.
\end{equation*}
 Therefore   by \eqref{union} for every $\varepsilon>0$ we can write 
\begin{equation}\label{fin1}
\#R_{k}=\sum_{g^2 |  k}\#R'_{k/g^2} \leq d(|k|)C_{\varepsilon}|k|^{\varepsilon}\sqrt{|\Delta|}\leq C_{\varepsilon}|k|^{\varepsilon}\sqrt{|\Delta|}.
\end{equation}

If $\Delta$ is zero or positive nonsquare, there can be  infinitely many automorphs.  But  the automorphs are generated in a very structured way, and  this will allow us to extract information about the structure of representations. We  lay forth the theory of automorphs for the case $\Delta$ a positive nonsquare integer.

We further  assume that $q$ is primitive, the non-primitive case  follows thereafter.  Since $\Delta$ is positive and non-square, it is at least $5$.  By Theorem 87 of \cite{ld}, a matrix 
\[\begin{bmatrix}
	    \alpha   & \beta \\
	\gamma & \delta  
\end{bmatrix}\]
is an automorph of $q$ if and only if 
\[\alpha=(t-bu)/2, \ \ \beta=-cu \ \ \gamma=au, \ \ \delta=(t+bu)/2\] 
where $t,u$ are integer solutions of $t^2-\Delta u^2=4$. Solutions of this equation turn out to be very structured. We observe that  $(\pm 2,0)$ are solutions for this equation, and that there can be no solution with $t=0$. By Theorem 88 of \cite{ld}, the equation have a solution $(t,u)$ with $t\neq 0,u\neq 0$. If $(t,u)$  is a solution, then $(-t,u),(t,-u),(-t,-u)$ are also  solutions. Therefore there is a solution with both entries positive.   If  $(t,u),(t',u')$ are two such solutions then  $t<t'$ implies $u<u'.$ Therefore there has to be a  solution $(T,U)$ with both entries positive and  minimum. This solution  is called the least positive  solution of the equation. Since we have $\Delta \geq 5 $ we must have $U\geq 1, T \geq 3.$  By Theorem 89 of \cite{ld}, all solutions $(t,u)$ of $t^2-\Delta u^2=4$ are given by 
\begin{equation}\label{autin1}
(t+\sqrt{\Delta}u)=i 2^{-j+1}(T+\sqrt{\Delta}U)^j, \quad  i=\pm1, \quad j=0,\pm 1,\pm 2 \ldots 
\end{equation}
and the automorphs corresponding to these solutions are given by
\begin{equation}\label{autin2}
 i[A]^j, \quad i=\pm1, \quad j=0,\pm 1,\pm 2 \ldots 
\end{equation}
with $[A]$ being the automorph corresponding to $(T,U)$. The only solutions $(t,u)$ with one entry zero are $(\pm2,0)$, and these come from the cases $i=\pm1,j=0$, and   if $(t,u)$ is a solution with both entries positive, then  $i,j$ corresponding to this solution are both positive. In this case $(-t,-u)$ is obtained by taking $-i,j$; $(t,-u)$ by $i,-j$; and $(-t,u)$ by $-i,-j$. Since $(T+\sqrt{\Delta}U)/2 >2$, different pairs of $i,j$ cannot give the same pair $(t,u)$. Therefore for each pair of $i,j$ we get a different solution from \eqref{autin1}.  We observe that no two   pairs $(t,u)\neq (t',u')$ can give rise to the same automorph, for since $a,c$ are both nonzero, we immediately get   different automorphs if $u\neq u'$, and if $u=u'$, then we must have $t\neq t'$, and this leads to different terms on the diagonals of corresponding automorph matrices.


\section{The case of positive definite forms}
In this section we will prove the first theorem. We will further the investigation of the second section with the tools obtained from analytic, geometric and arithmetic study of positive definite quadratic forms in sections 3 and 4. Specifically we will use Lemma 1 and \eqref{fin1} to prove the estimate \eqref{t1}. Then we will  provide, using arithmetic properties of specific forms and summability properties of certain functions, counterexamples showing the sharpness part of the  theorem.
\begin{proof}
If $p=\infty$ then, 
\begin{equation*}
\begin{aligned}
\|\mathcal{I}_{\lambda}f\|_{\infty}=\sup_{n\in \mathbb{Z}}\Big|\sum_{m\in \mathbb{Z}_*} \frac{f(q(m,n))}{|m|^{\lambda }}\Big|    \leq \|f\|_{\infty}\sum_{m\in \mathbb{Z}_*} \frac{1}{|m|^{\lambda }}  = C_{\lambda}\|f\|_{\infty}.
\end{aligned}
\end{equation*}
As is seen clearly, in this case our constant is entirely independent of the form $q$.

For the case  $p=1$  we will continue the analysis in section 2. We note that since the form is positive definite, the sets $A_k$  need to be considered only for $k\in \mathbb{N}$.
We partition the sets $A_k$ introduced there as follows
\[A_k':=\{(m,n)\in A_k : |m|\leq |k|^{1/4}(-\Delta)^{-1/2}   \}, \ \quad \quad  A_k'':=A_k\setminus A_k'.  \]
Then the quantity \eqref{eq5} satisfies by Lemma 1 

\begin{equation*}
\begin{aligned}
\sum_{(m,n)\in A_k} \frac{1}{|m|^{\lambda }}&=\sum_{(m,n)\in A_k'} \frac{1}{|m|^{\lambda }}+\sum_{(m,n)\in A_k''} \frac{1}{|m|^{\lambda }} \leq 4 + \sum_{(m,n)\in A_k''} \frac{1}{|m|^{\lambda }}.
\end{aligned}
 \end{equation*}
We estimate the cardinality of the set $A_k''$  by taking  $\varepsilon$ in \eqref{fin1}  to be $\lambda/8$. We then have 
\begin{equation*}
\begin{aligned}
 \sum_{(m,n)\in A_k''} \frac{1}{|m|^{\lambda }}\leq C_\lambda k^{\lambda/8}\sqrt{|\Delta|} k^{-\lambda/4}|\Delta|^{\lambda/2} = C_{\lambda}|\Delta|^{(\lambda+1)/2}k^{-\lambda/8}\leq C_{\lambda,\Delta}.
\end{aligned}
 \end{equation*}
Therefore
\begin{equation}\label{rev11}
\sum_{(m,n)\in A_k} \frac{1}{|m|^{\lambda }}\leq C_{\lambda,\Delta},
\end{equation}
and we conclude that 
\[  \|\mathcal{I}_{\lambda}f\|_1 \leq C_{\lambda,\Delta}\|f\|_1. \]

We now handle the case $1<p<\infty.$ Broadly, we proceed as in the case $p=1$, only significant difference being use of Hölder inequality to make utilization of the sets $A_k$ possible. We let $\lambda'=\lambda-1+p^{-1}$. Then

\begin{equation*}
\begin{aligned}
\|\mathcal{I}_{\lambda}f\|_{p}^p=\sum_{n\in \mathbb{Z}}\Big|\sum_{m\in \mathbb{Z}_*} \frac{f(q(m,n))}{|m|^{\lambda }}\Big|^p \leq \sum_{n\in \mathbb{Z}}\Big(\sum_{m\in \mathbb{Z}_*} \frac{|f(q(m,n))|}{|m|^{\lambda'/2 }}\frac{1}{|m|^{1-p^{-1}+\lambda'/2}}\Big)^p.
\end{aligned}
\end{equation*}
We  apply the Hölder inequality to bring the exponent $p$ inside the paranthesis
 
\begin{equation*}
\begin{aligned}
\leq \sum_{n\in \mathbb{Z}}\Big[\sum_{m\in \mathbb{Z}_*} \frac{|f(q(m,n))|^p}{|m|^{\lambda'p/2 }}\Big]\Big[\sum_{m\in \mathbb{Z}_*}\frac{1}{|m|^{1+\lambda'p/2(p-1)}}\Big]^{p-1}.
\end{aligned}
\end{equation*}
The second sum over $m$ is bounded by a constant $C_{p,\lambda}$ while  for the first sum  we  bring in the sets $A_k$. Thus we obtain
\begin{equation*}
\begin{aligned}
\leq C_{p,\lambda}\sum_{k\in \mathbb{N}} |f(k)|^p \sum_{(m,n)\in A_k} \frac{1}{|m|^{\lambda'p/2 }}.
\end{aligned}
\end{equation*}
 By  \eqref{rev11} the inner sum   is bounded by a constant $C_{p,\lambda, \Delta}$.
Therefore we conclude that
\begin{equation*}
\begin{aligned}
\|\mathcal{I}_{\lambda}f\|_{p}\leq C_{p,\lambda, \Delta}\|f\|_p.
\end{aligned}
\end{equation*}

We  turn to the sharpness part of the theorem. When $p=\infty$, it is clear that we cannot take $\lambda=1$, we can observe this just by taking $f$ to be a nonzero constant function and $q$ any positive definite form. 

For  $p=1$, we  first consider the case $r=1$. We will describe how to generalize this to the case of arbitrary $r$ afterwards.  We let $q(m,n)=m^2+n^2$. A well known theorem of Jacobi tells us that if $k$ is an odd   natural number for  which all prime factors  are of the form $4l+1$, then the number of representations of $k$ as a sum of two squares is $4d(k)$, see \cite{ld}. Therefore the set $A_k$ contains at least $4d(k)-2\geq 2d(k)$ elements. Thus if we consider the numbers $k_j:=(5\cdot 13)^j, \ j\in \mathbb{N}$, then $\# A_{k_j} \geq 2(j+1)^2$. But $j=\log k_j /\log 65 $, therefore 
  \[\# A_{k_j} \geq  \frac{2}{\log^2 65}\log^2k_j.\]  
 We define
 \[f(k):=\begin{cases}{j^{-2}} & \text{if } k=k_j  \\ 0 & \text{otherwise.}\end{cases}\]
 This $f$ clearly is in $l^1(\mathbb{Z})$. But we have
 
    \begin{equation*}
    \begin{aligned}
    \|\mathcal{I}_{log}f\|_1=\sum_{n\in \mathbb{Z}}\sum_{m\in \mathbb{Z}_*} \frac{f(q(m,n))}{\log(1+|m|)}&=\sum_{k\in \mathbb{N}}f(k)\sum_{(m,n)\in A_k} \frac{1}{\log(1+|m|)}\\ &= \sum_{j\in \mathbb{N}}f(k_j)\sum_{(m,n)\in A_{k_j}} \frac{1}{\log(1+|m|)}\\ &\geq \sum_{j\in \mathbb{N}}j^{-2} \frac{2}{\log^2 65}\log^2k_j \frac{1}{\log k_j} 
    \\ & =  \frac{2}{\log  65}\sum_{j\in \mathbb{N}}j^{-1}, 
    \end{aligned}
    \end{equation*}
which clearly is divergent. 
This process generalizes in a straightforward way  to the case of arbitrary $r$   by  taking a larger number of primes of the form $4l+1$  instead of just $5,13$. 

For the case $1<p<\infty$  we take $q(m,n):=m^2+n^2$, and 
\begin{equation}\label{p1}
f(k):=\begin{cases}{j^{-\frac{1}{p}} \log^{-\frac{1+p}{2p}}j } & \text{if } k=j^2, \ \ j\in \mathbb{N}-\{1\}  \\ 0 & \text{otherwise.}\end{cases}
\end{equation}
We then have 
 \begin{equation}\label{p2}
    \begin{aligned}
    \|\mathcal{I}_{1-p^{-1}}f\|_p^p=\sum_{n\in \mathbb{Z}}\Big|\sum_{m\in \mathbb{Z}_*} \frac{f(q(m,n))}{|m|^{1-p^{-1}}}\Big|^p
    &\geq \Big|\sum_{m\in \mathbb{Z}_*} \frac{f(q(m,0))}{|m|^{1-p^{-1}}}\Big|^p\\     
     &= \Big(2\sum_{m\geq 2} \frac{m^{-p^{-1}}\log^{-\frac{1+p}{2p}}m}{m^{1-p^{-1}}}\Big)^p\\  &= \Big(2\sum_{m\geq 2} \frac{1}{m\log^{\frac{1+p}{2p}}m}\Big)^p,          
    \end{aligned}
    \end{equation}
and this clearly diverges.

\end{proof}

 This proof suggests that the proper space for the study of $\mathcal{I}_{\lambda}$ is $l^1(\mathbb{Z})$, for the case $l^{\infty}(\mathbb{Z})$ is trivial, and the case $l^{p}(\mathbb{Z}),\  1<p<\infty$ follows largely  from the $l^1(\mathbb{Z})$ case and summability arguments.  The last counterexample makes it plain that  when the operators $\mathcal{I}_{\lambda}$ are applied to $l^{p}(\mathbb{Z})$ functions the sum in \eqref{fig} may not even  be finite unless $\lambda >1-p^{-1}$, and we do not even need to make use of    more delicate properties of the form $q$ to see this. 


\section{The case of indefinite forms}

We prove Theorem 2 in this section. The proof will rely  but also greatly expand upon ideas laid forth in the proof of Theorem 1. Most importantly instead of using the estimate \eqref{fin1} on the number of representations we will use their sparsity. We will establish this sparsity from \eqref{autin2} for the special case when the form $q$ is primitive and diagonal, and we will reduce the other cases to this case.

\begin{proof}
When $p=\infty $ the proof is exactly the same as it is for the corresponding case   of Theorem 1, and we obtain 

\[ \|\mathcal{I}_{\lambda}f\|_{\infty} \leq C_{\lambda}\|f\|_{\infty}.\]

Let $p=1$. We first  assume $q(m,n)=am^2+cn^2$ with $a>0,c<0$ and $\gcd(a,c)=1$. Since the determinant $-4ac$ must be  nonsquare it is at least 8.  We again consider the sets $A_k, k \in \mathbb{Z}$. If $k=0$, then it follows from \eqref{lines} that $A_k$ is empty, therefore we assume  $k \neq  0$. In this case we define 
\begin{equation*}
 A_{k}':=\{(m,n)\in \mathbb{Z}_*\times \mathbb{Z}_*:q(m,n)=k \},
\end{equation*}
and since $A_k\setminus A_k'$ can contain at most 2 elements, we have 
 \begin{equation*}
 \sum_{(m,n)\in A_k} \frac{1}{|m|^{\lambda }}\leq 2+  \sum_{(m,n)\in A_k'} \frac{1}{|m|^{\lambda }}. 
 \end{equation*}
If  $(m,n)\in A_k',$ then  $(-m,-n),(m,-n),(-m,n)\in A_k'$. Therefore if we define
\begin{equation*}
 A_{k}'':=\{(m,n)\in \mathbb{N}\times \mathbb{N}:q(m,n)=k \},
\end{equation*}
 we have
  \begin{equation*}
  \sum_{(m,n)\in A_k} \frac{1}{|m|^{\lambda }}\leq 2+  \sum_{(m,n)\in A_k'} \frac{1}{|m|^{\lambda }}=2+4\cdot\sum_{(m,n)\in A_k''} \frac{1}{|m|^{\lambda }}. 
  \end{equation*}
 Henceforth we  concentrate our efforts on estimating the sum over $A_k''$.  We decompose
\begin{equation*}
A_k''=\bigcup_{ g^2|k}A_{k,g}'', \ \ \ \  \ \ A_{k,g}'':=\{(m,n)\in \mathbb{N}\times \mathbb{N}:q(m,n)=k,\ \gcd(m,n)=g \}.
\end{equation*}
Then, as laid out in section 4, the map $(m,n)\mapsto (m/g,n/g)$ gives a bijection from $A_{k,g}''$ onto $A_{k/g^2,1}''$, and each  representation in $A_{k/g^2,1}''$  emerges from a solution of  
\begin{equation}\label{cong21}
u_g^2\equiv \Delta \ \  (\text{mod}\ 4|k/g^2|), \quad \text{and}  \quad 0\leq u_g< 2|k/g^2|,
\end{equation}
 and the corresponding form $[k/g^2,u_g,v_g]$. Let therefore 
\begin{equation*}
A_{k/g^2,1}''=\bigcup_{u_g}A_{k/g^2,1,u_g}''.
\end{equation*}
  These decompositions  induce decompositions of $A_{k,g}''$ into  subsets $ A_{k,g,u_{g}}''$. We thus write
 \begin{equation*}\label{}
 \sum_{(m,n)\in A_k''} \frac{1}{|m|^{\lambda }}=  \sum_{g^2|k} \    \sum_{u_{g}} \ \   \sum_{(m,n)\in A_{k,g,u_{g}}''} \frac{1}{|m|^{\lambda }}. 
 \end{equation*}
We have
\begin{equation*}\label{}
   \sum_{(m,n)\in A_{k,g,u_{g}}''} \frac{1}{|m|^{\lambda }} =\frac{1}{g^{\lambda}}\sum_{(m,n)\in A_{k/g^2,1,u_{g}}''} \frac{1}{|m|^{\lambda }}. 
 \end{equation*}
  The elements of $A_{k/g^2,1,u_{g}}''$, if it is nonempty,  are obtained from  multiplying a matrix  that takes $q$ to $[k/g^2,u_{g},v_{g}]$ by the automorphs of $q$. Let $[U_{g}]$  
 be a matrix taking $q$ to $[k/g^2,u_{g},v_{g}]$, and let $[A]$ be the automorph of $q$ corresponding to the least positive solution $(T,U)$ of the equation $t^2-\Delta u^2=4$, thus
 \begin{equation}\label{matr15}
[U_g]:= \begin{bmatrix}
  	    \alpha_{u_g}   & \beta_{u_g} \\
  	\gamma_{u_g} & \delta_{u_g}  
  \end{bmatrix}, \ \ \ \ [A]:=\begin{bmatrix}
   	    T/2   & -cU \\
   	aU & T/2  
   \end{bmatrix}.
  \end{equation}
 We  have $T^2=\Delta U^2+4\geq \Delta+4\geq 12$, and therefore $T\geq 2\sqrt{3}$.  The elements $(m,n)\in A_{k/g^2,1,u_{g}}''$ are the first columns of  the  matrices in the chains below for which both entries of the first column are  positive      
\begin{equation}\label{cha}
\begin{aligned}
\ldots  [A]^{-2}[U_{g}], \ \ [A]^{-1}[U_{g}], \ \ &[A]^0[U_{g}], \ \ [A][U_{g}], \ \ [A]^2[U_{g}]  \ldots \\
\ldots   -[A]^{-2}[U_{g}],  \ -[A]^{-1}[U_{g}],  \ -&[A]^0[U_{g}],  \ -[A][U_{g}],  \ -[A]^2[U_{g}]  \ldots
\end{aligned}
\end{equation}
  We let 
\begin{equation*}
\begin{bmatrix}
 	    \alpha   & \beta \\
 	\gamma & \delta  
 \end{bmatrix}
 \end{equation*}
stand for an arbitrary matrix in these chains. Since we have  
\begin{equation*}
\begin{bmatrix}
 	    T/2   & -cU \\
 	aU & T/2  
 \end{bmatrix} \begin{bmatrix}
 	    \alpha   & \beta \\
 	\gamma & \delta  
 \end{bmatrix}=
 \begin{bmatrix}
 	    T\alpha/2 -cU\gamma   & T\beta/2-cU\delta \\
 	aU\alpha+T\gamma/2 &  aU\beta+T\delta/2  
 \end{bmatrix},
 \end{equation*}
if $\alpha,\gamma> 0$, then   $T\alpha/2 -cU\gamma,  aU\alpha+T\gamma/2> 0$, and further 
\begin{equation}\label{sqrt3}
T\alpha/2 -cU\gamma>   \sqrt{3}\alpha.      
\end{equation}
Therefore if there is a matrix in one of the chains \eqref{cha}  for which 
\begin{equation}\label{pp}
\alpha>0, \ \ \ \ \ \gamma> 0,
\end{equation}
 then for every matrix to the right of it the same property holds. This immediately implies that if there are matrices with this property in  these two  chains, all of them must be located in just one chain. On this chain owing to \eqref{sqrt3} there must be a leftmost matrix with this property. Let  
 \begin{equation*}
[U_g']:= \begin{bmatrix}
  	    \alpha_{u'_g}   & \beta_{u'_g} \\
  	\gamma_{u'_g} & \delta_{u'_g}  
  \end{bmatrix}
  \end{equation*}
 denote this matrix. Thus all elements of $A_{k/g^2,1,u_{g}}''$ come from the first columns of $[A]^j[U_g']$ for $j\geq 0$, and the first entry of  the first column for $[A]^j[U_g']$  is not less than $3^{j/2}\alpha_{u'_g}$. 
 Therefore 
 \begin{equation*}\label{}
 \sum_{(m,n)\in A_{k/g^2,1,u_{g}}''} \frac{1}{|m|^{\lambda }}\leq \frac{1}{\alpha_{u'_g}^{\lambda}}\sum_{j= 0}^{\infty} \frac{1}{3^{j\lambda/2}}\leq \frac{C_{\lambda}}{\alpha_{u'_g}^{\lambda}}, 
  \end{equation*}
 and thus
 \begin{equation*}\label{}
    \sum_{(m,n)\in A_{k,g,u_{g}}''} \frac{1}{|m|^{\lambda }} \leq\frac{C_{\lambda}}{(g\alpha_{u'_g})^{\lambda}}, 
  \end{equation*}
with $(g\alpha_{u'_g},g\gamma_{u'_g})\in A_{k,g,u_{g}}''$ being just one representation. Thus we succeeded in estimating an infinite sum over representations of a certain type by just one representation of that type.  Then 
\begin{equation*}\label{}
 \sum_{(m,n)\in A_k''} \frac{1}{|m|^{\lambda }}=  \sum_{g^2|k} \    \sum_{u_{g}} \ \   \sum_{(m,n)\in A_{k,g,u_{g}}''} \frac{1}{|m|^{\lambda }} \leq  C_{\lambda} \sum_{g^2|k} \    \sum_{u_{g}}\frac{1}{(g\alpha_{u'_g})^{\lambda}}. 
 \end{equation*}
There  are at most 4 representations $q(m,n)=k$ with $|m|\leq |k|^{1/4}\Delta^{-1/2}$ by Lemma 2, therefore we have 
\begin{equation}\label{rev12}
 \leq C_{\lambda}\Big[4+   \frac{{\Delta}^{\lambda/2}}{|k|^{\lambda/4}}\sum_{g^2|k} \    \sum_{u_{g}} 1  \Big]. 
 \end{equation}
The double sum gives the sum over $g$ of number of solutions of 
\eqref{cong21}, and as explained in section 4, for each $g$ this number is bounded by
 $d(|k|)\sqrt{\Delta}.$
 The number of possible $g$ is again bounded by $d(|k|)$. Therefore if we use the estimate $d(|k|)\leq C_{\lambda}|k|^{\lambda/16}$, the double sum is bounded by $C_{\lambda}|k|^{\lambda/8}\sqrt{\Delta}.$ Hence
  \begin{equation*}
   \leq C_{\lambda}\Big[4+   \frac{{\Delta}^{\lambda/2}}{|k|^{\lambda/4}} C_{\lambda}|k|^{\lambda/8}\sqrt{\Delta}\Big]\leq C_{\lambda}\big[4+    C_{\lambda}\Delta^{(1+\lambda)/2}\big] \leq    C_{\lambda}\Delta^{(1+\lambda)/2}. 
   \end{equation*}
Therefore
\begin{equation}\label{mn7}
 \sum_{(m,n)\in A_k} \frac{1}{|m|^{\lambda }}\leq C_{\lambda}{\Delta}^{(1+\lambda)/2}, 
 \end{equation}
and this suffices to conclude that  for $q(m,n)=am^2+cn^2$ primitive with $a>0,c<0$ we have
\begin{equation*}\label{}
\|\mathcal{I}_{\lambda}f\|_{1} \leq C_{\lambda}{\Delta}^{(1+\lambda)/2}\|f\|_1. 
 \end{equation*}

We  move to the general case. Let $q(m,n)=am^2+bmn+cn^2$ be an indefinite form of nonsquare discriminant. Discriminant being nonsquare implies $a\neq 0,c\neq 0$.  Then, 
\begin{equation*}\label{}
\|I_{\lambda,q}f\|_{1} \leq \sum_{k\in \mathbb{Z}} |f(k)|  \sum_{(m,n)\in A_k} \frac{1}{|m|^{\lambda }} , 
 \end{equation*}
with $A_k:=\{(m,n)\in \mathbb{Z}_*\times \mathbb{Z}:am^2+bmn+cn^2=k \}$. But we have 
\begin{equation*}
\begin{aligned}
&\{(m,n)\in \mathbb{Z}_*\times \mathbb{Z}:am^2+bmn+cn^2=k \}\\
=&\{(m,n)\in \mathbb{Z}_*\times \mathbb{Z}:4acm^2+4bcmn+4c^2n^2=4ck \}\\
=&\{(m,n)\in \mathbb{Z}_*\times \mathbb{Z}:(bm+2cn)^2-\Delta(q)m^2=4ck \}
\\
=&\{(m,n)\in \mathbb{Z}_*\times \mathbb{Z}:\Delta(q)m^2 -(bm+2cn)^2=-4ck .\}
\end{aligned}
\end{equation*}
We define the form $q'(x,y)=\Delta(q)x^2-y^2$, and the sets $A_{q',k}:=\{(m,n)\in \mathbb{Z}_*\times \mathbb{Z}:\Delta(q)m^2-n^2=k \}$. Then we have $\Delta(q')=4\Delta(q)$. If $(m,n)\in A_k$, then $(m,bm+2cn)\in A_{q',-4ck}$, and the map $(m,n)\mapsto (m,bm+2cn)$
is injective. Therefore
\begin{equation}\label{mn8}
 \begin{aligned}
\sum_{(m,n)\in A_k} \frac{1}{|m|^{\lambda }}&\leq \sum_{(m,n)\in A_{q',-4ck}} \frac{1}{|m|^{\lambda }}.
\end{aligned}
 \end{equation}
Since the form $q'$ is of a type covered by our investigation above, from \eqref{mn7} with $\Delta(q')=4\Delta(q)$ we  have
\begin{equation}\label{mn9}
 \sum_{(m,n)\in A_{q',-4ck}} \frac{1}{|m|^{\lambda }}\leq C_{\lambda}2^{1+\lambda}{\Delta(q)}^{(1+\lambda)/2}=C_{\lambda}{\Delta(q)}^{(1+\lambda)/2}, 
 \end{equation}
and thus
\begin{equation*}\label{}
\|I_{\lambda,q}f\|_{1} \leq C_{\lambda}{\Delta(q)}^{(1+\lambda)/2} \|f\|_1= C_{\lambda,\Delta(q)} \|f\|_1.
 \end{equation*}

We  let $1<p<\infty$, and $q(m,n)=am^2+bmn+cn^2$ an indefinite form of nonsquare discriminant. Proceeding exactly as in the positive definite case yields

\begin{equation*}
\begin{aligned}
\|\mathcal{I}_{\lambda}f\|_{p}^p\leq  C_{p,\lambda}\sum_{k\in \mathbb{N}} |f(k)|^p \sum_{(m,n)\in A_k} \frac{1}{|m|^{\lambda'p/2 }}
\end{aligned}
\end{equation*}
with $\lambda'=\lambda-1+p^{-1}.$ Combining \eqref{mn8} and \eqref{mn9} yields

\begin{equation*}
\begin{aligned}
 \sum_{(m,n)\in A_k} \frac{1}{|m|^{\lambda'p/2 }} \leq C_{p,\lambda}\Delta^{(1+\lambda'p/2)/2},
\end{aligned}
\end{equation*}
and therefore
\begin{equation*}
\begin{aligned}
\|\mathcal{I}_{\lambda}f\|_{p}\leq  C_{p,\lambda}\Delta^{1/2p+\lambda'/4}\|f\|_p=C_{p,\lambda,\Delta}\|f\|_p .
\end{aligned}
\end{equation*}

We move to the sharpness part of the theorem.  We observe that the claim is clear for
 $p=\infty$.
Let $p=1$. If  $r=1$, the unboundedness  is easy to show using the infinitude of automorphs.  
 Indeed, let $q(m,n):=m^2-8n^2$ and 
\[f(k)=\begin{cases}{1} & \text{if } k=4  \\ 0 & \text{otherwise.}\end{cases}\]
We have representations of $4$ given by \eqref{autin1}, which in our case becomes
 \begin{equation*}
 (t+2\sqrt{2}u)=i 2(3+2\sqrt{2})^j, \quad  i=\pm1, \quad j=0,\pm 1,\pm 2 \ldots. 
 \end{equation*}
It suffices to consider the representations $t_j,u_j>0$ obtained when $i=1$ and $j>0.$ 
These satisfy $t_j\leq 2\cdot 6^j.$ Hence
 \begin{equation*}
    \begin{aligned}
    \|\mathcal{I}_{log}f\|_1=\sum_{k\in \mathbb{N}}f(k)\sum_{(m,n)\in A_k} \frac{1}{\log(1+|m|)} &= \sum_{(m,n)\in A_{4}} \frac{1}{\log(1+|m|)}\\ &\geq\sum_{j\in \mathbb{N}} \frac{1}{\log(1+t_j)} \\ &\geq \sum_{j\in \mathbb{N}} \frac{1}{\log 6^{j+1}}, 
    \end{aligned}
    \end{equation*}
and this diverges. But this method clearly does not generalize to the cases 
 $r\geq 2$.  For these cases we must use the arithmetic structure of the quadratic forms, just as we did in section 4. But we will also need further effort  to deal with indefiniteness of the form. More specifically, in section 4 the form we used $q(m,n)=m^2+n^2$ allows us to conclude that if $q(m,n)=k$,  then $|m|\leq |k|$. No such conclusion is possible for indefinite forms, and indeed we know that $|m|\geq C|k|$ is possible for any $C\in \mathbb{N}$. However    we will be able to find for  each solution  of \eqref{cong21} for which there are corresponding representations one  representation $(m,n)$  of $k$ with $|m|\leq 10|k|$, and this will suffice. Thus we will not make use of the infinitude of representations, but we must calculate exactly the number of solutions of \eqref{cong21}, and make sure that each solution gives rise to representations of $k$. 

  We take the form $q(m,n):=m^2-2n^2$. Therefore  $\Delta(q)=8$. We take  the primes $7,17$, which are of the  form $8l\pm 1$.  Then we consider $k_j=(7\cdot17)^{2j+1},\  j\in \mathbb{N}.$ We define
\[f(k)=\begin{cases}{j^{-2}} & \text{if } k=k_j  \\ 0 & \text{otherwise.}\end{cases}\]
We compute the number of solutions of \eqref{cong21}.
If  $g$ is a square divisor of $k_j$, it must have the form $7^{i_1}17^{i_2}, \ 0\leq i_1,i_2 \leq j$, and thus the number of solutions of \eqref{cong21}  is 
\[\Gamma_8(7^{2(j-i_1)+1} 17^{2(j-i_2)+1})=\Gamma_8(7^{2(j-i_1)+1})\Gamma_8( 17^{2(j-i_2)+1}).\]
Both factors on the right hand side are 2 by Theorem 17 of \cite{ld}. Therefore  \eqref{cong21} has $4$ solutions for each choice of $i_1,i_2$. Since there is a total of $j+1$ choices for each of $i_1,i_2$, the equations \eqref{cong21} have in total $4(j+1)^2\geq \log^2k_j/\log^2 119 $ solutions. For each solution $[k_j/g^2,u_g,v_g]$ we want to find a matrix $[U_g]$ 
    of determinant 1 that maps $q$ to this solution, as given by \eqref{matr15}. Existence of such a matrix is guaranteed, for both $q$ and $[k_j/g^2,u_g,v_g]$ have discriminant 8, and there is only one equivalence class of forms for this discriminant; see page 99-104 of \cite{ld}. Since $\alpha_{u_g}^2-2\gamma_{u_g}^2=k_j/g^2>0$, we have $\alpha_{u_g}\neq 0$. Since the matrix $-[U_g]$ also have the desired properties of $[U_g]$, we may assume $\alpha_{u_g}>0.$  All representations corresponding to $u_g$ are  then as in \eqref{cha}, and   
 the automorph $[A]$ and its inverse is given in our case by
\[[A]:=\begin{bmatrix}
  	    3   &  4\\ 
  	2 &  3  
  \end{bmatrix},  \quad \quad \quad  [A]^{-1}= \begin{bmatrix}
    	    3   &  -4\\ 
    	-2 &  3  
    \end{bmatrix}.  \]
As mentioned above, we want to associate just one representation to $u_g$. If $\alpha_{u_g}\leq 10k_j/g^2$, then we let $(g\alpha_{u_g},g\gamma_{u_g})$ be this representation. We now suppose   $\alpha_{u_g}> 10k_j/g^2$. Therefore $\gamma_{u_g}^2>49k_j/g^2$, and $\gamma_{u_g}$ is either positive or negative. If it is positive, then  from the matrix $[A]^{-1}[U_g]$ we obtain the representation $(3\alpha_{u_g}-4\gamma_{u_g},-2\alpha_{u_g}+3\gamma_{u_g})$, for which we have 
$0<3\alpha_{u_g}-4\gamma_{u_g}<\alpha_{u_g} /3$, and $0<-2\alpha_{u_g}+3\gamma_{u_g}$. If $3\alpha_{u_g}-4\gamma_{u_g}\leq 10k_j/g^2$ we multiply this representation  with $g$ and  associate it to $u_g$, if not, we repeat the process. Clearly in a finite number of steps we obtain a desired representation. Similarly, if $\gamma_{u_g}$ is negative,  from the matrix $[A][U_g]$ we obtain the representation $(3\alpha_{u_g}+4\gamma_{u_g},2\alpha_{u_g}+3\gamma_{u_g})$, for which we have 
$0<3\alpha_{u_g}+4\gamma_{u_g}<\alpha_{u_g} /3$, and $2\alpha_{u_g}+3\gamma_{u_g}<0$. If $3\alpha_{u_g}+4\gamma_{u_g}\leq 10k_j/g^2$ we  multiply this representation by  $g$ and associate it to $u_g$, if not, we repeat the process until, in a finite number of steps, we obtain a representation with desired properties.

Therefore for any solution $u_g$ we have a representation $(g\alpha_{u_g},g\gamma_{u_g})$ where $0<\alpha_{u_g}\leq 10k_j/g^2$. Having obtained these representations we can write
 \begin{equation*}
    \begin{aligned}
    \|\mathcal{I}_{log}f\|_1=\sum_{n\in \mathbb{Z}}\sum_{m\in \mathbb{Z}_*} \frac{f(q(m,n))}{\log(1+|m|)}&=\sum_{k\in \mathbb{N}}f(k)\sum_{(m,n)\in A_k} \frac{1}{\log(1+|m|)}\\ &= \sum_{j\in \mathbb{N}}f(k_j)\sum_{(m,n)\in A_{k_j}} \frac{1}{\log(1+|m|)}\\ &\geq \sum_{j\in \mathbb{N}}j^{-2} \frac{\log^2k_j}{\log^2 119} \frac{1}{2\log k_j} 
    \\ & \geq  \frac{1}{\log 119}\sum_{j\in \mathbb{N}}j^{-1}, 
    \end{aligned}
    \end{equation*}
and this is clearly divergent.  This example can easily be generalized to  any $r\in  \mathbb{N}$. The number $8$ is a quadratic residue for   any prime $p$ of the form $8l\pm1$, see chapter 3 of \cite{ld}, and this allows us to conclude that $u_g^2\equiv 8 \ (\text{mod}\  p)$ have 2 solutions. Therefore if we take more primes of this form instead of just $7,17$ the same method allows us to show unboundedness  for any $r\in  \mathbb{N}$.

 When $1<p<\infty$ and $\lambda=1-p^{-1}$, we take $q(m,n):=m^2-2n^2$, and $f$ as in \eqref{p1}. Then the steps in \eqref{p2} yield the desired result.

\end{proof}

\end{document}